\documentclass[a4paper,12pt]{article}
\usepackage{amsmath, amsthm, amstext}
\usepackage{amssymb, array, amsfonts}
\usepackage[utf8]{inputenc}
\usepackage{graphicx}

\numberwithin{equation}{section}

\def \ga{\gamma}
\def \de{\delta}

\def \la{\lambda}

\def \oo{\omega}

\def \G{\Gamma}
\def \D{\Delta}

\def \O{\Omega}

\def \C{\mathbb{C}}
\def \N{\mathbb{N}}
\def \R{\mathbb{R}}
\def \Z{\mathbb{Z}}

\def\n{\nabla}
\def\dd{\partial}

\def\1{1\!\!\!\!1}

\def\mes{\operatorname{mes}}

\def\supp{\operatorname{supp}}

\theoremstyle{plain}
\newtheorem{theorem}{\bf Theorem}[section]
\newtheorem{lemma}[theorem]{\bf Lemma}

\theoremstyle{definition}

\theoremstyle{remark}
\newtheorem{rem}[theorem]{\bf Remark}

\renewcommand{\le}{\leqslant}
\renewcommand{\ge}{\geqslant}

\renewcommand{\qed}{\vrule height7pt width5pt depth0pt}

\title{On the P\'olya conjecture for the Neumann problem in planar convex domains}

\author{N.~Filonov
\thanks{The work is supported by the grant of Russian Science Foundation No. 22-11-00092.}}
\date{}
\begin{document}
\maketitle

\begin{abstract}
Denote by $N_{\cal N} (\O,\la)$ the counting function of the spectrum
of the Neumann problem in the domain $\O$ on the plane.
G.~P\'olya conjectured that 
$N_{\cal N} (\O,\la) \ge (4\pi)^{-1} |\O| \la$.
We prove that for convex domains 
$N_{\cal N} (\O,\la) \ge (2 \sqrt 3 \,j_0^2)^{-1} |\O| \la$.
Here $j_0$ is the first zero of the Bessel function $J_0$.
\footnote{Keywords: P\'olya conjecture, Neumann problem, convex domains.}
\end{abstract}

\section{Formulation of the result}
Let $\O \subset \R^d$ be a bounded domain with Lipschitz boundary.
We consider the Dirichlet and Neumann problems for the Laplace operator in $\O$,
$$
\begin{cases}
- \D u = \la u \ &\text{in} \ \ \O, \\
u = 0 \  &\text{on} \ \ \dd\O,
\end{cases}
\qquad
\begin{cases}
- \D v = \mu v \ &\text{in} \ \ \O, \\
\frac{\dd v}{\dd\nu} = 0 \ &\text{on} \ \ \dd\O.
\end{cases}
$$
It is well known that the spectra of the both problems are discrete.
Denote by $\la_k$ and $\mu_k$ the corresponding eigenvalues taking multiplicity into account,
\begin{equation*}
0 < \la_1 < \la_2 \le \la_3 \le \dots, \qquad \la_k \to + \infty, 
\end{equation*}
\begin{equation*}
0 = \mu_1 < \mu_2 \le \mu_3 \le \dots, \qquad \mu_k \to + \infty,
\end{equation*}
Introduce also the counting functions
$$
N_{\cal D} (\O,\la) := \# \{ k: \la_k \le \la\}, \qquad
N_{\cal N} (\O,\la) := \# \{ k: \mu_k \le \la\} .
$$
G.~P\'olya in his book \cite{P1} conjectured that the estimates
\begin{equation*}
N_{\cal D} (\O,\la) \le \frac{|B_1| |\O|}{(2\pi)^d}\ \la^{d/2},
\end{equation*}
\begin{equation}
\label{11}
N_{\cal N} (\O,\la) \ge \frac{|B_1| |\O|}{(2\pi)^d}\ \la^{d/2}
\end{equation}
hold true for all domains $\O$ and for all $\la \ge 0$.
Here $|\O|$ denotes the volume of the set $\O$,
and $B_1$ is the unit ball in $\R^d$.
Note that the coefficient in front of $\la^{d/2}$ 
coincides with the coefficient in the Weyl asymptotics.

We list the known results on the P\'olya conjecture for the Neumann case:
\begin{itemize}
\item in 1961, P\'olya himself proved \cite{P2} the estimate \eqref{11}
for regular tiling domains. 
The domain $\O$ is called tiling if the whole space $\R^d$ can be covered by 
non-intersecting copies of $\O$ up to a set of measure zero;
the domain is regular tiling if the corresponding covering is periodic;
\item in 1966, Kellner proved \cite{Kellner} the estimate \eqref{11} for all tiling domains;
\item in 1992, Kr\" oger proved \cite{Kr} the inequality
\begin{equation}
\label{12}
N_{\cal N} (\O,\la) \ge \frac2{d+2} \cdot \frac{|B_1| |\O|}{(2\pi)^d}\ \la^{d/2}
\end{equation}
for all domains $\O$ and all $\la \ge 0$.
\item In 2022, the estimate \eqref{11} was proved \cite{FLPS} for the disk 
and for circular sectors of arbitrary aperture in the plane.
\end{itemize}

In this paper we prove the following

\begin{theorem}
\label{main}
Let $\O \subset \R^2$ be a convex bounded domain.
Then
\begin{equation}
\label{13}
N_{\cal N} (\O, \la) \ge \frac{|\O| \la}{2 \sqrt 3 \,j_0^2}.
\end{equation}
Here and everywhere below we denote by $j_\nu$ the first positive root
of the Bessel function $J_\nu$. 
In particular, $j_0 \approx 2,4048$.
\end{theorem}

Note that in 2D case the P\'olya conjecture \eqref{11} and the Kr\" oger estimate \eqref{12}
take the form
$$
N_{\cal N} (\O, \la) \ge \frac{|\O| \la}{4 \pi}
\qquad \text{and} \qquad 
N_{\cal N} (\O, \la) \ge \frac{|\O| \la}{8 \pi}
$$
respectively. 
We have
$$
\frac1{8\pi} \approx 0,0398, \quad \frac1{2 \sqrt 3 \,j_0^2} \approx 0,0499,
\quad \frac1{4\pi} \approx 0,0796.
$$
Thus, the coefficient in \eqref{13} is better than the coefficient in \eqref{12},
but we prove \eqref{13} only for convex domains.

\begin{rem}
In terms of the eigenvalues themselves in the two-dimensional case 
the inequalities \eqref{11}, \eqref{12} and \eqref{13} read as follows:
$$
\mu_{l+1} \le \frac{4\pi l}{|\O|}, \qquad 
\mu_{l+1} \le \frac{8\pi l}{|\O|}, \qquad 
\text{and} \qquad
\mu_{l+1} \le \frac{2 \sqrt 3 \,j_0^2 \,l}{|\O|}
$$
respectively.
\end{rem}

\section{Lemmas}

\begin{lemma}
\label{l21}
Let $J_\nu$ be the Bessel function of order $\nu \ge 0$.
Then
$$
\int_0^s \left(\left(t^{-\nu} J_\nu(t)\right)'\right)^2 t^{2\nu+1} dt 
\le \int_0^s J_\nu (t)^2 t\,dt
\qquad \text{for all} \quad s \in [0, j_\nu].
$$
\end{lemma}

\begin{proof}
Integrating by parts we obtain
\begin{equation}
\label{21}
\int_0^s \left(\left(t^{-\nu} J_\nu(t)\right)'\right)^2 t^{2\nu+1} dt
= \left.\left(t^{-\nu} J_\nu(t)\right)' t^{\nu+1} J_\nu(t)\right|_0^s
- \int_0^s t^{-\nu} J_\nu(t) \left(\left(t^{-\nu} J_\nu(t)\right)' t^{2\nu+1}\right)' dt.
\end{equation}

Further,
$$
\left(t^{-\nu} J_\nu(t)\right)' t^{2\nu+1}
= t^{\nu+1} J_\nu'(t) - \nu t^\nu J_\nu(t),
$$
\begin{equation}
\label{22}
\left(\left(t^{-\nu} J_\nu(t)\right)' t^{2\nu+1}\right)'
= t^{\nu+1} J_\nu''(t) + t^\nu J_\nu'(t) - \nu^2 t^{\nu-1} J_\nu(t) 
= - t^{\nu+1} J_\nu(t)
\end{equation}
due to the Bessel equation.
We have $J_\nu(t) \ge 0$ on $[0,j_\nu]$,
therefore the right hand side of \eqref{22} is non-positive,
and the function $\left(t^{-\nu} J_\nu\right)' t^{2\nu+1}$ decreases on $[0,j_\nu]$.
On the other hand,
$$
\left.\left(t^{-\nu} J_\nu\right)' t^{2\nu+1}\right|_{t=0} = 0,
$$
so
$$
\left(t^{-\nu} J_\nu\right)' t^{2\nu+1} < 0 \quad \text{for} \quad 0 < t \le j_\nu.
$$
This means that the first term in the right hand side of \eqref{21} is non-positive.
Now, \eqref{21} and \eqref{22} imply
$$
\int_0^s \left(\left(t^{-\nu} J_\nu(t)\right)'\right)^2 t^{2\nu+1} dt
\le - \int_0^s t^{-\nu} J_\nu(t) \left(\left(t^{-\nu} J_\nu(t)\right)' t^{2\nu+1}\right)' dt
= \int_0^s J_\nu(t)^2 t \, dt.
\quad \qedhere
$$
\end{proof}

\begin{lemma}
\label{l22}
Let $r>0$.
Introduce notations
\begin{equation}
\label{23}
\ga_1 = (2r; 0), \quad \ga_2 = (r; \sqrt 3 \, r),
\qquad \G = \left\{\ga = n_1 \ga_1 + n_2 \ga_2\right\}_{n_1, n_2 \in \Z},
\end{equation}
$\G$ is a regular triangular lattice in the plane.
If $\ga, \tilde \ga \in \G$, $\ga \neq \tilde \ga$, then $|\ga - \tilde \ga| \ge 2r$.
\end{lemma}

This Lemma is obvious.

\begin{lemma}
\label{l23}
Let $\O \subset \R^2$ be a measurable set of finite measure, $|\O| < \infty$.
Let $r>0$.
Then there is a vector $b \in \R^2$ such that
$$
\# \left(\O \cap (\G+b)\right) \ge \frac{|\O|}{2\sqrt 3\, r^2},
$$
where the lattice $\G$ is defined in \eqref{23}, 
and $\G+b = \{\ga+b\}_{\ga\in\G}$ is the shifted lattice.
\end{lemma}

\begin{proof}
Denote by ${\cal O}$ a cell of $\G$,
$$
{\cal O} = \left\{t_1 \ga_1 + t_2 \ga_2\right\}_{t_1, t_2 \in [0,1)}.
$$
Clearly, 
$|{\cal O}| = 2 \sqrt 3 \,r^2$.
We have
$$
\# \left(\O \cap (\G+b)\right) = \sum_{\ga \in \G} \chi_\O (\ga + b),
$$
where $\chi_\O$ is the characteristic function of $\O$.
So,
$$
\int_{\cal O} \# \left(\O \cap (\G+b)\right) db 
= \int_{\cal O} \sum_{\ga \in \G} \chi_\O (\ga + b) \,db
= \int_{\R^2} \chi_\O (y) \, dy = |\O|.
$$
Therefore, there is a vector $b \in {\cal O}$ such that
$$
\# \left(\O \cap (\G+b)\right) \ge \frac{|\O|}{|{\cal O}|}.
\qquad \qedhere
$$
\end{proof}

\section{Proof of Theorem \ref{main}}
In the recent paper K.~Funano proved the following inequality.

\begin{theorem}[\cite{KF}, Lemma 3.1]
\label{t31}
Let $\O\subset \R^d$ be a convex bounded domain.
Let $r>0$, $l \in \N$.
Assume that there are points $x_1, \dots, x_l \in \O$ such that
$$
|x_j-x_k| \ge 2r \quad \text{if} \quad j \neq k.
$$
Then the $l$-th eigenvalue of the Neumann problem in $\O$ satisfies the estimate
$$
\mu_l \le c_0\,d^2\,r^{-2},
$$
where $c_0$ is an absolute constant.
\end{theorem}

We refine this inequality.

\begin{theorem}
\label{t32}
Under the assumptions of Theorem \ref{t31} we have
$$
\mu_l \le j_{\frac{d}2-1}^2 \,r^{-2}.
$$
\end{theorem}

\begin{proof}
Introduce the function
$$
F(\rho) = \rho^{1-d/2} J_{\frac{d}2-1} \left(\frac{\rho \,j_{\frac{d}2-1}}r\right), 
\qquad \rho > 0,
$$
and define
$$
f_k (x) = \begin{cases}
F\left(|x-x_k|\right), & \text{if} \ |x-x_k| < r,\\
0, & \text{if} \ |x-x_k| \ge r,
\end{cases}
\qquad k = 1, \dots, l.
$$
Clearly,
$f_k \in W_2^1 (\O)$ and
$$
\n f_k(x) =  \begin{cases}
F'\left(|x-x_k|\right) \cdot \frac{x-x_k}{|x-x_k|}, & \text{if} \ |x-x_k| < r,\\
0, & \text{if} \ |x-x_k| \ge r.
\end{cases}
$$
The intersection of the convex domain $\O$ with a ball is also convex.
It can be described in spherical coordinates as
$$
B_r(x_k) \cap \O = \left\{ x = x_k + (\rho; \oo) :
\oo \in S^{d-1}, 0 \le \rho < R_k(\oo)\right\},
$$
where $S^{d-1}$ denotes the unit sphere in $\R^d$, 
and $R_k$ is a continuous function on $S^{d-1}$,
$$
0 < R_k(\oo) \le r \qquad \forall \ \oo.
$$
We have
\begin{eqnarray}
\nonumber
\int_\O |f_k(x)|^2 dx 
= \int_{S^{d-1}} dS(\oo) \int_0^{R_k(\oo)} F(\rho)^2 \rho^{d-1} d\rho \\
= \int_{S^{d-1}} dS(\oo) \int_0^{R_k(\oo)} 
J_{\frac{d}2-1} \left(\frac{\rho j_{\frac{d}2-1}}r\right)^2 \rho \, d \rho 
\label{31} \\
= \left(\frac{r}{j_{\frac{d}2-1}}\right)^2 
\int_{S^{d-1}} dS(\oo) \int_0^{R_k(\oo)\, j_{\frac{d}2-1} r^{-1}} 
J_{\frac{d}2-1} (t)^2 \,t \, dt,
\nonumber
\end{eqnarray}
where we maked the change of variables
\begin{equation}
\label{32}
\rho = \frac{r \,t}{j_{\frac{d}2-1}} .
\end{equation}
On the other hand,
\begin{eqnarray}
\nonumber
\int_\O |\n f_k(x)|^2 dx 
= \int_{S^{d-1}} dS(\oo) \int_0^{R_k(\oo)} F'(\rho)^2 \rho^{d-1} d\rho \\
= \int_{S^{d-1}} dS(\oo) \int_0^{R_k(\oo)} 
\left(\frac{d}{d\rho} 
\left(\rho^{1-d/2} J_{\frac{d}2-1} \left(\frac{\rho j_{\frac{d}2-1}}r\right)\right)\right)^2 
\rho^{d-1} \, d \rho 
\label{33} \\
= \int_{S^{d-1}} dS(\oo) \int_0^{R_k(\oo)\, j_{\frac{d}2-1} r^{-1}} 
\left(\frac{d}{dt} \left(t^{1-d/2} J_{\frac{d}2-1} (t)\right)\right)^2 t^{d-1} \, dt,
\nonumber
\end{eqnarray}
where we maked the same change \eqref{32}.
Lemma \ref{l21} with $\nu = \frac{d}2-1$ yields
\begin{equation}
\label{34}
\int_0^{R_k(\oo)\, j_{\frac{d}2-1} r^{-1}} 
\left(\frac{d}{dt} \left(t^{1-d/2} J_{\frac{d}2-1} (t)\right)\right)^2 t^{d-1} \, dt
\le \int_0^{R_k(\oo)\, j_{\frac{d}2-1} r^{-1}} J_{\frac{d}2-1} (t)^2 \,t \, dt,
\end{equation}
as $R_k(\oo) \le r$.
Now, \eqref{31}, \eqref{33} and \eqref{34} imply the inequality
$$
\int_\O |\n f_k(x)|^2 dx \le \left(\frac{j_{\frac{d}2-1}}r \right)^2 
\int_\O |f_k(x)|^2 dx.
$$
Introduce the space of linear combinations of $f_k$
$$
{\cal L} := \left\{f(x) = \sum_{k=1}^l c_k f_k(x)\right\}_{c_k \in \C}.
$$
By construction,
$$
\mes \left(\supp f_j \cap \supp f_k\right) = 0 \qquad \text{if} \quad j \neq k.
$$
Therefore,
$$
\int_\O |\n f(x)|^2 dx \le \left(\frac{j_{\frac{d}2-1}}r \right)^2 
\int_\O |f(x)|^2 dx \qquad \forall f \in {\cal L}.
$$
As $\dim {\cal L} = l$ the claim follows.
\end{proof}

{\it Proof of Theorem \ref{main}.}
Let $\la > 0$.
Put $r = \frac{j_0}{\sqrt \la}$.
By virtue of Lemma \ref{l22} and Lemma \ref{l23} one can pick out 
points $x_1, \dots, x_l$ in $\O$ such that
$$
|x_j-x_k| \ge 2r \quad \text{if} \quad j \neq k, 
\qquad \text{and} \quad l \ge\frac{|\O|}{2\sqrt 3\, r^2}.
$$
Theorem \ref{t32} implies now
$\mu_l \le j_0^2\, r^{-2}$, and therefore,
$$
N_{\cal N} (\O, \la) \ge l \ge\frac{|\O|}{2\sqrt 3\, r^2} = 
\frac{|\O| \la}{2 \sqrt 3 \,j_0^2}. \qquad \qed
$$

\begin{rem}
Let $\G_d$ be a lattice in $\R^d$ such that $|\ga - \tilde \ga| \ge 2$
for all $\ga, \tilde \ga \in \G_d$, $\ga \neq \tilde \ga$.
Denote by ${\cal O}_d$ a cell of $\G_d$.
In the same manner as above we obtain that for any convex bounded domain 
$\O\subset \R^d$
\begin{equation}
\label{*}
N_{\cal N} (\O,\la) \ge \frac{|\O| \,\la^{d/2}}{|{\cal O}_d| (j_{\frac{d}2-1})^d}.
\end{equation}
On the other hand it is easy to see that
\begin{equation}
\label{**}
\frac{|B_1|}{|{\cal O}_d|} \le \de_d
\end{equation}
where $\de_d$ is the optimal sphere packing density in $\R^d$.
The exact value of $\de_d$ is known today for $d = 1, 2, 3, 8$ and $24$ only.
The value of $\de_d$ in other dimensions is a famous open question,
see for example \cite{Cohn} and references therein.
The estimate \eqref{**} and the known estimate \cite{Lev}
$$
\de_d \le \frac{(j_{\frac{d}2})^d}{2^{2d} \,\G \left(\frac{d+2}2\right)^2}
$$
imply that the coefficient in the right hand side of \eqref{*} satisfies
$$
\frac1{|{\cal O}_d| (j_{\frac{d}2-1})^d} \le \frac{\de_d}{|B_1| (j_{\frac{d}2-1})^d}
< \frac{2|B_1|}{(d+2) (2\pi)^d} \quad \text{if} \quad d \ge 3.
$$
So, the bound \eqref{*} does not improve the Kr\" oger bound \eqref{12} for $d \ge 3$.
\end{rem}


\end{document}